\documentclass[11pt]{amsart}

\usepackage{amsmath,amssymb,graphicx}
\usepackage{verbatim}
\usepackage{caption}
\usepackage{subcaption}
\usepackage{graphicx, amssymb, enumerate, pstricks, pst-node}
\usepackage{pst-all}
\usepackage{epsfig}
\usepackage{epic}
\usepackage{url}
\usepackage{latexsym}
\usepackage{amsmath}
\usepackage{amsfonts}
\usepackage{amssymb}
\usepackage{amssymb,stmaryrd}
\usepackage{tikz,graphics}
\usepackage{tikz}
\usepackage{mathdots}
\usepackage{xcolor}
\newtheorem{theorem}{Theorem}[section]
\newtheorem{lemma}[theorem]{Lemma}

\newtheorem{proposition}[theorem]{Proposition}

\newtheorem{sublemma}{}[theorem]

\theoremstyle{definition}

\newtheorem{example}[theorem]{Example}
\theoremstyle{remark}

\numberwithin{equation}{section}
\newcommand{\del}{\backslash}

\DeclareMathOperator{\simp}{si}
\nonstopmode
\begin{document}

\title[On the cogirth of binary matroids]{On the cogirth of binary matroids}
\author{Cameron Crenshaw}
\address{Mathematics Department\\
Louisiana State University\\
Baton Rouge, Louisiana, USA}
\email{ccrens5@lsu.edu}

\author{James Oxley}
\address{Mathematics Department\\
Louisiana State University\\
Baton Rouge, Louisiana, USA}
\email{oxley@math.lsu.edu}


%\subjclass[2010]{05B35, 05C15}

\date{\today}

\begin{abstract}
The cogirth, $g^\ast(M)$, of a matroid $M$ is the size of a smallest cocircuit of $M$. Finding the cogirth of a graphic matroid can be done in polynomial time, but Vardy showed in 1997 that it is NP-hard to find the cogirth of a binary matroid. In this paper, we show that $g^\ast(M)\leq \frac{1}{2}\vert E(M)\vert$ when $M$ is binary, unless $M$ simplifies to a projective geometry. We also show that, when equality holds, $M$ simplifies to a Bose-Burton geometry, that is, a matroid of the form $PG(r-1,2)-PG(k-1,2)$. These results extend to matroids representable over arbitrary finite fields.
\end{abstract}

\maketitle

\section{Introduction}

For an arbitrary graph $G$, the well-known fact that the degree sum of $G$ is twice the number of edges of $G$ implies that
\begin{equation*}
\frac{\vert E(G)\vert}{\delta(G)}\geq \frac{1}{2}\vert V(G)\vert,
\end{equation*}
where $\delta(G)$ is the minimum degree of $G$. In a matroid $M$ of nonzero rank, the \emph{cogirth}, $g^\ast(M)$, of $M$ is the size of a smallest cocircuit of $M$. As $U_{r,n}$ shows, $\frac{\vert E(M)\vert}{g^\ast(M)}$ can be arbitrarily close to 1 even for simple matroids, although it is bounded below by $\frac{1}{2}(r(M)+1)$ when $M$ is graphic.

In this paper, we show that, when $M$ is binary,
\begin{equation*}
\frac{\vert E(M)\vert}{g^\ast(M)}\geq 2
\end{equation*}
unless $M$ simplifies to a projective geometry. We also characterize the matroids that achieve equality in this bound. Both of these results are special cases of results for matroids representable over arbitrary finite fields.

The terminology used here will follow Oxley~\cite{oxley} with the following addition. We will often use $P_r$ and $A_r$ to denote $PG(r-1,q)$ and $AG(r-1,q)$, respectively, where $q$ should be clear from the context. The next two results are the main results of the paper.

\begin{theorem}
\label{notpgintro}
For $r\geq 1$, let $M$ be a rank-$r$ matroid representable over $GF(q)$ whose simplification is not $P_r$. Then $$\frac{\vert E(M)\vert}{g^\ast(M)}\geq \frac{q}{q-1}.$$
Moreover, equality holds if and only if $M$ is loopless and, for a fixed embedding of $\simp(M)$ in $P_r$,
\begin{itemize}
\item[(i)]{the complement of $\simp(M)$ in $P_r$ is isomorphic to $P_k$ for some $k$ with $1\leq k < r$; and}
\item[(ii)]{if $P$ is a copy of $P_{k+1}$ in $P_r$ containing the complement of $\simp(M)$, then the parallel classes of the elements in $E(M)\cap E(P)$ all have the same size; and}
\item[(iii)]{$\vert E(N) \vert\geq (q-1)\vert E(M)-E(N)\vert$ for every restriction $N$ of $M$ that simplifies to $A_r$.}
%replaced VVV with ^^^(iii)'
%\item[(iii)]{each parallel class of $M$ has at most $\tfrac{\vert E(M)\vert}{q^{r-1}}$ elements.}
%\item[(ii)']{if $\ell$ is a line of $PG(r-1,q)$ meeting $M$ and its complement, then the parallel classes of $M$ on $\ell$ all have the same size.}
\end{itemize}
\end{theorem}

This theorem excludes the matroids $M$ for which $\simp(M)\cong PG(r-1,q)$. These excluded matroids are covered by the next result.

\begin{proposition}
\label{pgintro}
For $r\geq 1$, let $M$ be a matroid that simplifies to $PG(r-1,q)$. Then $$\frac{\vert E(M)\vert}{g^\ast(M)}\geq\frac{q^r-1}{q^{r-1}(q-1)}.$$
Moreover, equality holds if and only if $M$ is loopless and all its parallel classes have the same size.
\end{proposition}

Condition (i) in Theorem~\ref{notpgintro} says that $M$ simplifies to a \emph{Bose-Burton geometry}~\cite{BB}, that is, a matroid that is obtained from $PG(r-1,q)$ by deleting some $PG(k-1,q)$ where $1\leq k<r$. In each of our results, the bound on $\tfrac{\vert E(M)\vert}{g^\ast(M)}$ is relatively easy to obtain. The core of each proof involves characterizing when equality holds in the bound. The proofs appear in Section~\ref{proofs}.

\section{Preliminaries}
\label{prelims}

In a matroid $M$ of rank at least one, a loop contributes to $\vert E(M)\vert$ but not to $g^\ast(M)$. Since our concern here is on bounding $\tfrac{\vert E(M)\vert}{g^\ast(M)}$ below, we shall focus on matroids without loops. It will be convenient here to deal with the parallel classes in such a matroid $M$ by assigning, to each element of $\simp(M)$, a \emph{weight} $w(e)$ that is equal to the cardinality of the parallel class of $M$ that contains $e$. Thus we deal with simple matroids with associated weight functions that take a positive-integer value on each element. For a set $X$ in such a matroid $N$, we write $w(X)$ for $\sum_{x\in X}w(x)$ and write $w(N)$ for $w(E(N))$. The weight function of $N\del Y$ is the restriction of the weight function of $N$ to $E(N)-Y$. When $Y$ is contracted from $N$, we replace each parallel class $P$ by a single element $e_P$ whose weight in the contraction is $w_N(P)$. We will call this weighted simple matroid the \emph{weighted contraction} of $Y$ and denote it by $N/Y$, even though the underlying matroid is actually $\simp(N/Y)$. The cogirth of a weighted matroid is the minimum weight of a cocircuit.

\section{The Proofs} 
\label{proofs}

We begin with a lemma that serves as the base case for both of the inductive arguments that prove the inequalities in the main results.

\begin{lemma}
\label{rank2}
Let $M$ be a simple, rank-$2$ matroid representable over $GF(q)$, and let $w$ be a weight function on $M$. Then $$\frac{w(M)}{g^\ast(M)}\geq \frac{\vert E(M)\vert}{\vert E(M)\vert-1},$$
with equality if and only if $w$ is constant.
\end{lemma}

\begin{proof}
Let $w_1\leq w_2\leq \cdots\leq w_n$ be the weights of the elements of $M$. As $g^\ast(M)=w_1+w_2+\cdots+w_{n-1}$, the desired inequality is equivalent to
\begin{equation*}
(n-1)(w_1+w_2+\cdots+w_n)\geq n(w_1+w_2+\cdots+w_{n-1}).
\end{equation*}
Subtracting $(n-1)(w_1+w_2+\cdots+w_{n-1})$ from each side, we obtain
\begin{equation}
\label{rank2eqtn}
(n-1)w_n\geq w_1+w_2+\cdots+w_{n-1},
\end{equation}
which is true since $w_n\geq w_i$ for all $i$. Note that equality holds in~(\ref{rank2eqtn}) if and only if $w_i=w_n$ for all $i$.
\end{proof}

The following is the main result of the paper. It is equivalent to Theorem~\ref{notpgintro} and is stated here in terms of weights.

\begin{theorem}
\label{notpgweight}
Let $M$ be a simple, rank-$r$ matroid representable over $GF(q)$, and let $w$ be a weight function on $M$. Suppose $M\not\cong P_r$. Then $$\frac{w(M)}{g^\ast(M)}\geq \frac{q}{q-1}.$$
Moreover, equality holds if and only if, for a fixed embedding of $M$ in $P_r$,
\begin{itemize}
\item[(i)]the complement of $M$ is isomorphic to $P_k$, with $1\leq k<r$; and
\item[(ii)]{if $P$ is a copy of $P_{k+1}$ containing the complement of $M$ in $P_r$, then $w$ is constant on $P$; and}
\item[(iii)]{$w(N)\geq (q-1)w(E(M) - E(N))$ for every $A_r$-restriction $N$ of $M$.}
%\item[(iii)]{$q^{r-1}w(e)\leq w(M)$ for all $e$ in $E(M)$.}
%\item[(ii)]if $\ell$ is a line of $P_r$ meeting $M$ and its complement, then $w$ is constant on $\ell$.
\end{itemize}
\end{theorem}

\begin{proof}
We begin by proving the displayed inequality by induction on $r$. Lemma~\ref{rank2} gives the result when $r=2$, so suppose $r\geq 3$. If there is an $e$ in $E(M)$ with $M/e\not\cong P_{r-1}$, then, by induction,
\begin{equation*}
(q-1)w(M)>(q-1)w(M/e)\geq qg^\ast(M/e)\geq qg^\ast(M).
\end{equation*}
Thus we may assume that $M/e\cong P_{r-1}$ for all $e$ in $E(M)$. Take a line of $P_r$ that meets both $E(M)$ and $E(P_r)-E(M)$. Let $X$ be the set of elements of $M$ on this line and $e$ be a maximum-weight element of $X$. Let $Y=X-e$. Note that $\vert Y\vert\leq q-1$, so
\begin{equation}
\label{littlenotpg}
w(Y)\leq w(e)(q-1).
\end{equation}
Observe that $M\del Y/e$ has rank $r-1$ but is not isomorphic to $P_{r-1}$ so, by the induction assumption,
\begin{equation*}
(q-1)w(M\del Y/e)\geq qg^\ast(M\del Y/e).
\end{equation*}
Now
\begin{equation*}
w(M\del Y/e)=w(M)-w(e)-w(Y),
\end{equation*}
and
\begin{equation*}
g^\ast(M\del Y/e)\geq g^\ast(M)-w(Y).
\end{equation*}
Thus
\begin{equation*}
(q-1)w(M)\geq qg^\ast(M)+w(e)(q-1)-w(Y),
\end{equation*}
so, by~(\ref{littlenotpg}),
\begin{equation*}
(q-1)w(M)\geq qg^\ast(M)
\end{equation*}
as desired.

Next we characterize when equality is achieved in the last bound. Let $M^c$ be the complement of the fixed embedding of $M$ in $P_r$. When $M^c\cong P_k$ for $1\leq k<r$, a hyperplane of $P_r$ either contains this $P_k$ or meets it in a $P_{k-1}$. Thus a cocircuit of $M$ is isomorphic to either $A_r$ or $A_r-A_k$. We call these \emph{type-I} and \emph{type-II cocircuits}, respectively, noting that there are no type-I cocircuits when $k=r-1$.

\begin{sublemma}
\label{typeiilemma}
Suppose $M$ satisfies (i) and (ii). If $C^\ast$ is a type-II cocircuit of $M$, then $$w(C^\ast)=\frac{q-1}{q}w(M).$$
\end{sublemma}

As $M$ satisfies (i), $M^c\cong P_k$. Since $C^\ast$ is a type-II cocircuit, there is a restriction $A$ of $P_r$ isomorphic to $A_r$ such that $A$ meets $M^c$ and $C^\ast=E(M)\cap E(A)$. Let $H$ be the hyperplane of $P_r$ that is the complement of $A$.

Now, $P_r$ consists of $\frac{q^{r-k}-1}{q-1}$ copies of $P_{k+1}$ containing $M^c$, and the pairwise intersection of these copies is $M^c$. Thus $M$ is the disjoint union of $\frac{q^{r-k}-1}{q-1}$ copies of $A_{k+1}$. By (ii), the elements in each $A_{k+1}$ have the same weight. To complete the proof of~\ref{typeiilemma}, we show that $C^\ast$ contains exactly $\frac{q-1}{q}$ of the elements of each $A_{k+1}$.

Consider the complementary $A_{k+1}$ to $M^c$ in a fixed $P_{k+1}$. Note that $P_{k+1}$ consists of $q+1$ copies of $P_k$, including $M^c$, that contain $H\cap E(M^c)$, which is isomorphic to $P_{k-1}$. Therefore, this $A_{k+1}$ is the disjoint union of $q$ copies of $A_k$. Now $H$ meets $P_{k+1}$ at a $P_k$ distinct from $M^c$. Thus $A$ meets $P_{k+1}$ in a set that is the union of $q$ disjoint copies of $A_k$, one of which is in $M^c$. This implies that $C^\ast\cap A_{k+1}$ is the disjoint union of $q-1$ copies of $A_k$, and~\ref{typeiilemma} follows.

Now assume that $\tfrac{w(M)}{g^\ast(M)}=\tfrac{q}{q-1}$. Then equality holds in~(\ref{littlenotpg}) so $\vert Y\vert=q-1$ and $w(y)=w(e)$ for all $y\in Y$. The former implies that every line of $P_r$ that meets both $M$ and $M^c$ contains exactly $q$ points of $M$. This means that every line that contains two points of $M^c$ lies entirely in $M^c$. Thus $M^c$ is a flat of $P_r$, proving (i).

As $w(y)=w(e)$ for all $y\in Y$, it follows that $w$ is constant on each line of $P_r$ that meets both $M$ and $M^c$. Since a $P_k$ contained in a $P_{k+1}$ meets every line of the $P_{k+1}$, (ii) is satisfied. It now follows from~\ref{typeiilemma} that $\tfrac{w(M)}{g^\ast(M)}=\frac{q}{q-1}$ if and only if $M$ satisfies (i) and (ii), and the type-I cocircuits of $M$ have weight at least $\tfrac{q-1}{q}w(M)$. It is straightforward to check that this third condition is equivalent to (iii), so the theorem holds.
\end{proof}

The reader may find condition (iii) of Theorem~\ref{notpgweight} unsatisfying, and the next proposition offers a potential replacement, (iii)$^\prime$. The example that follows Proposition~\ref{extracondition} shows that conditions (i), (ii), and (iii)$^\prime$ do not guarantee $\frac{w(M)}{g^\ast(M)}=\frac{q}{q-1}$ for a matroid $M$ meeting the hypotheses of Theorem~\ref{notpgweight}. In addition, the example illustrates the potential difficulty of finding a satisfactory replacement for (iii).

\begin{proposition}
\label{extracondition}
Let $M$ be a simple, rank-$r$ matroid representable over $GF(q)$, and let $w$ be a weight function on $M$. Suppose that $M\not\cong P_r$ and that $\tfrac{w(M)}{g^\ast(M)}=\frac{q}{q-1}$. Then
\begin{itemize}
\item[(iii)$^\prime$]{$q^{r-1}w(e)\leq w(M)$ for all $e$ in $E(M)$.}
\end{itemize}
\end{proposition}

\begin{proof}
By Theorem~\ref{notpgweight}, $M^c=P_k$. If $k=r-1$, then Theorem~\ref{notpgweight}(ii) implies that (iii)$^\prime$ holds with equality. Thus we may assume that $k<r-1$. Extend the weight function of $M$ to $P_r$ by assigning each element of $M^c$ a weight of one. Then contract $M^c$ from $P_r$ to form a weighted matroid $M'\cong P_{r-k}$. Fix an element $e$ in $E(M)$, and let $e'$ be the image of $e$ in $M'$. Note that $w(M')=w(M)$ and $w(e')=q^kw(e)$ under this transformation. Moreover, the type-I cocircuits of $M$ correspond to the cocircuits of $M'$, so the weight of each cocircuit of $M'$ equals the weight of the corresponding type-I cocircuit of $M$.

Let $C^\ast$ be a cocircuit of $M'$ that avoids $e'$ and let $H$ be the complementary hyperplane to $C^\ast$ in $M'$. Since $\frac{w(M)}{g^\ast(M)}=\frac{q}{q-1}$, it follows that
\begin{equation*}
\frac{q}{q-1}w(C^\ast)\geq w(M'),
\end{equation*}
and subtracting $w(C^\ast)$ from each side produces
\begin{equation}
\label{repeatme}
\frac{1}{q-1}w(C^\ast)\geq w(H).
\end{equation}

Note that~(\ref{repeatme}) holds for an arbitrary cocircuit of $M'$ avoiding $e'$, so we have such an inequality for every such cocircuit. Moreover, $C^\ast$ and $H$ partition $E(M')$ so, for a fixed $f\in E(M'\del e')$, its weight contributes to exactly one side of each inequality. Now, there are $\frac{q^{r-k-1}-1}{q-1}$ total inequalities as this is the number $t$ of hyperplanes of $M'$ containing $e'$. Similarly, $w(f)$ contributes to the right-hand side of exactly $\frac{q^{r-k-2}-1}{q-1}$ of these inequalities as this is the number $s$ of hyperplanes of $M'$ containing both $e'$ and $f$. Hence $w(f)$ contributes to the left-hand side of $t-s$ of these inequalities. Summing these inequalities gives
\begin{equation*}
\frac{t-s}{q-1}w(M'\del e')\geq sw(M'\del e')+tw(e'),
\end{equation*}
and this simplifies to
\begin{equation*}
w(M')\geq q^{r-k-1}w(e').
\end{equation*}
Finally, we substitute $w(M)$ for $w(M')$ and $q^kw(e)$ for $w(e')$ to obtain
\begin{equation*}
w(M)\geq q^{r-1}w(e)
\end{equation*}
as desired.
\end{proof}

\begin{example}
Let $q=2$ and let $M=P_4-p$ for some $p\in E(P_4)$. Then $M\cong P_4-P_1$. Take a hyperplane $H$ of $P_4$ containing $p$, and note that $H$ has $\vert P_3\vert$ elements. Then $H-p$ is a hyperplane of $M$ and the corresponding cocircuit $C^\ast$ is type-I and has $\vert A_4\vert$ elements.

Assign the weight 2 to each element of $H-p$ and the weight 1 to each element of $C^\ast$. Then
\begin{equation*}
w(M)=2(\vert P_3\vert-1)+1\cdot\vert A_4\vert=2(6)+8=20.
\end{equation*}
Observe that conditions (i) and (ii) of Theorem~\ref{notpgweight} hold and, since, for all $e$ in $E(M)$,
\begin{equation*}
q^{r-1}w(e)\leq 2^3(2)<20=w(M),
\end{equation*}
so does (iii)$^\prime$. However, $w(C^\ast)=8$, so the equation
\begin{equation}
\label{equality}
\frac{w(M)}{g^\ast(M)}=\frac{q}{q-1}
\end{equation}
fails.

Now, in $P_4$, take a line in $C^\ast\cup p$ and another in $H$ that each meet in $\{p\}$. Swap the weights 1 and 2 on the elements of $M$ on these lines. Note that (i), (ii), and (iii)$^\prime$ continue to hold, and $w(M)$ is unchanged. However, it is straightforward to check that the weights of the type-I cocircuits of $M$ are at least 10, so~(\ref{equality}) holds by Theorem~\ref{notpgweight}. Thus, characterizing the matroids for which equality holds in Theorem~\ref{notpgweight} requires not only restricting the weights themselves, but also controlling their distribution.
\end{example}

Finally, we prove a proposition equivalent to Proposition~\ref{pgintro} stated here in terms of weights.

\begin{proposition}
Let $M$ be a matroid isomorphic to $PG(r-1,q)$ and $w$ be a weight function on $E(M)$. Then $$\frac{w(M)}{g^\ast(M)}\geq\frac{q^r-1}{q^{r-1}(q-1)}.$$ Moreover, equality holds if and only if $w$ is constant.
\end{proposition}

\begin{proof}
We prove the inequality by induction on $r$. It is trivial when $r=1$ and is true for $r=2$ by Lemma~\ref{rank2}, so suppose $r\geq 3$. Let $C^\ast$ be a cocircuit of $M$ of weight $g^\ast(M)$ and let $H$ be the complementary hyperplane to $C^\ast$ in $M$. Choose $Z$ as a maximum-weight hyperplane of $H$. Then, letting $Y$ be the complement of $Z$ in $H$, we have
\begin{equation}
\label{fullpglong}
\frac{w(M)}{g^\ast(M)}=\frac{w(C^\ast)+w(Y)+w(Z)}{w(C^\ast)}=1+\frac{w(Y) + w(Z)}{w(C^\ast)}.
\end{equation}
Observe that the weighted contraction of $Z$ from $M$ is isomorphic to $P_2$ so, by the inequality for $r=2$, we get that
\begin{equation*}
\frac{w(M/Z)}{g^\ast(M/Z)}\geq \frac{q+1}{q}.
\end{equation*}
Now, since $C^\ast$ is also a minimum-weight cocircuit of $M/Z$, we rewrite this inequality as
\begin{equation}
\label{fullpgindr2}
\frac{w(Y)+w(C^\ast)}{w(C^\ast)}\geq \frac{q+1}{q}.
\end{equation}
It follows that $qw(Y)\geq w(C^\ast)$. Substituting into~(\ref{fullpglong}), we obtain
\begin{equation}
\label{fullpglong2}
\frac{w(M)}{g^\ast(M)}\geq 1+\frac{1}{q}\cdot\frac{w(Y) + w(Z)}{w(Y)}.
\end{equation}

Finally, the hyperplane $H$ is isomorphic to $P_{r-1}$, and our choice of $Z$ makes $Y$ a minimum-weight cocircuit of $H$. Thus, by induction,
\begin{equation}
\label{fullpgind}
\frac{w(Y) + w(Z)}{w(Y)}\geq \frac{q^{r-1}-1}{q^{r-2}(q-1)}.
\end{equation}
Substituting~(\ref{fullpgind}) into~(\ref{fullpglong2}) gives the desired inequality.

One easily checks that, when $w$ is constant,
\begin{equation}
\label{fullpgequality}
\frac{w(M)}{g^\ast(M)}=\frac{q^r-1}{q^{r-1}(q-1)}.
\end{equation}
We now use induction on $r$ to prove that the elements of $M$ have the same weight when~(\ref{fullpgequality}) holds. When $r=1$, this is trivial, and Lemma~\ref{rank2} handles the rank-2 case. 

Suppose $r\geq 3$. Since~(\ref{fullpgequality}) holds, equality holds in~(\ref{fullpgindr2}) and~(\ref{fullpgind}). It follows from the latter using the induction assumption that the weight function $w$ on $E(M)$ is constant on the hyperplane $H$ of $M$. From the former, we deduce that, in $M/Z$, every point has equal weight. It follows that every hyperplane $H'$ of $M$ containing $Z$ has the same weight. Hence the cocircuit $E(M)-H'$ has the same weight as $C^\ast$. Replacing $H$ by $H'$, we deduce that $w$ is constant on the elements of $H'$. Letting $H'$ range over all of the hyperplanes of $M$ containing $Z$, we deduce that $w$ is constant on $E(M)$.
\end{proof}

%\section*{Acknowledgements}

\end{document}